\documentclass[12pt]{amsart}
\usepackage{amscd,amssymb}
\usepackage[arrow,matrix]{xy}
\usepackage[plainpages,backref,urlcolor=blue]{hyperref}
\usepackage{enumerate}

\theoremstyle{plain}
\newtheorem{thm}[subsection]{Theorem}
\newtheorem{lem}[subsection]{Lemma}
\newtheorem{prop}[subsection]{Proposition}
\newtheorem{cor}[subsection]{Corollary}

\theoremstyle{definition}
\newtheorem{rk}[subsection]{Remark}

\newtheorem{ex}[subsection]{Example}

\numberwithin{equation}{section}
\newcommand{\A}{{\mathcal A}}

\newcommand{\LL}{{\mathcal L}}

\newcommand{\Z}{\mathbb{Z}}

\newcommand{\C}{\mathbb{C}}

\newcommand{\PP}{\mathbb{P}}

\begin{document}
%\date{June 4, 2009}

\title [On the monodromy of Milnor fibers of hyperplane arrangements]
{On the monodromy of Milnor fibers of hyperplane arrangements }

\author[Pauline Bailet]{Pauline Bailet}
\address{Univ. Nice Sophia Antipolis, CNRS,  LJAD, UMR 7351, 06100 Nice, France.}
\email{Pauline.BAILET@unice.fr}

\subjclass[2000]{Primary  32S22, 32S55; Secondary  32S25, 32S40}

\keywords{hyperplane arrangements, Milnor fiber, monodromy, local systems }

\begin{abstract} 
We describe a general setting where the monodromy action on the first cohomology group of the Milnor fiber of a hyperplane arrangement is the identity.

\end{abstract}

\maketitle

\section{Introduction}

Let $\A =\{H_1,\,...\,,H_d\} \subset \C^{n+1}$ be a central arrangement of $d$ hyperplanes, with Milnor Fiber $F_{\A},\,$ and intersection lattice $L(\A).$ For any edge $X \in L(\A),\,$ we note $\A_X:=\{ H \in \A\,\,|\,\, X \subset H\}$ the corresponding subarrangement.
We associate to $\A$ the projective arrangement $\A' \subset \PP_{\C}^{n}$ obtained by associating to a   hyperplane $H \in \A$, given by $\ell_H=0$, the hyperplane $H' \in \PP_{\C}^{n}$ 
defined by the same equation $\ell_H=0$.
We note $M(\A)$ and $M(\A')$ the complements of $\A$ and $\A'.$ \\
Consider the Orlik-Solomon algebra $A_R^*(\A)$ of $\A$ with coefficients in a unitary commutative ring $R,\,$ and the corresponding Aomoto complex:
$(A_R^*(\A),\omega_1 \wedge)$,
where $\omega_1= {\sum_{H \in \A} a_H} \in A_{R}^1(\A).$ Here $a_H \in A_{R}^1(\A)$ denotes the element of $A_{R}^1(\A)$ corresponding to the hyperplane $H$, see \cite{OT}.

\medskip

Let $\,\lambda=exp(2\sqrt{-1} \pi /d).$
For $q\geq 0,$ we denote by $H^q(F_{\A})_{\lambda^k}$ the  $\lambda^k$- eigenspace of the monodromy operator $h^q: H^q(F_{\A},\C) \to H^q(F_{\A},\C),$ for $0 \leq k \leq d-1.$
There is a  well known relation between these eigenspaces and the cohomology of $M(\A')$ with coefficients in a rank one local system \cite{CS}, \cite{ST}:  
\begin{center}
$H^q(F_{\A})_{\lambda^k}= H^q(M(\A'),\LL_{\lambda^k}),$
\end{center} 
where $\LL_{\lambda^k}$ is the rank one local system on $M(\A')$ whose monodromy around  any hyperplane of $\A'$ is $\lambda^k.$
The main result of this note, Theorem \ref{pg1} below,  is a vanishing result describing many situations where $H^q(F_{\A})_{\lambda^k}=0$ for $k \ne 0.$\\

Let us begin by introducing a new  combinatorial object associated to a hyperplane arrangement $\A$, namely a graph $G(\A)$ given by:
 \begin{itemize}
 \item The vertices of $G(\A)$ correspond to the hyperplanes of $\A.$
 \item Two different vertices $H_1$ and $H_2$ are linked by an edge (we will note $H_1-H_2$) if and only if $\A_X= \{H_1,H_2\},\,$ where $X=H_1 \cap H_2.$ 
 \end{itemize}
We say that such a graph is connected if for any two vertices $H_1$ and $H_2,$ we can find an edge sequence linking $H_1$ and $H_2.$ 

Note that in hyperplane arrangement theory it is a classical idea to associate an arrangement to a graph (to obtain a graphic arrangement). However, it seems that the converse construction of a graph from an arrangement is rather unexplored.

With the previous notation we have the following main result.

\begin{thm} \label{pg1} 
Suppose the following assumptions are verified.

\medskip

\noindent (i) The graph $G(\A)$ is connected.

\noindent (ii) For every codimension $2$ intersection $X$ of hyperplanes in $\A$, we have 
$|\A_X| \leq 9.$

\noindent (iii) We have either
$6 \nmid d,$ 
or there exists an hyperplane $H \in \A$ such that if $X$ is an intersection of hyperplanes of $\A$ of codimension $2,\,\,X \subset H,\,$ then $|\A_X| \neq 6.$

\medskip

Then $H^1(F_\A,\C)=H^1(F_\A)_1$, i.e. $H^1(F_{\A})_{\lambda^k}=0$ for $k \ne 0$.
\end{thm}

\begin{rk} 
(i) We have that 
$$H^1(F_\A)_1 = H^1(M(\A'),\LL_{\lambda^0})=H^1(M(\A'),\C) = \C^{d-1}.$$
Thus, Theorem \ref{pg1} gives large number of situations where  $H^1(F_\A,\C)$ is determined by the intersection lattice $L(\A).$
Indeed, the graph $G(\A)$ is constructed with the information given by $L(\A)$.
In general, the question whether the cohomology of $F_{\A}$ is determined by $L(\A)$ is still open, even if many advances have been made (see for instance  the results of A. Macinic and S. Papadima \cite{PM} for the first Betti number of graphic arrangements, as well as the results of M. Yoshinaga on real line arrangements \cite{Y1}, \cite{Y2} ).
\medskip

(ii) By taking a generic $3$-dimensional subspace $E\subset \C^{n+1} $ and replacing $\A'$ by the corresponding line arrangement in $\PP(E)=\PP^2$, we can consider from the beginning that $n=2$. This follows from the Zariski Theorem of Lefschetz type due to Hamm, Hamm-L\^e and Goreski-MacPherson, see for instance for the simplest version \cite{Di}, p.25. Moreover, in the case of a line arrangement, the action of $h^1$ determines all the actions $h^*$ in view of the usual formula for the zeta-function of the monodromy of the Milnor fiber of a homogeneous polynomial, see for instance \cite{Di}, p.107.\\

(iii) The  case where every distinct lines $H,\,H'$ of $\A$ are linked by an edge corresponds to the case where $\A$ is generic and then Theorem \ref{pg1} follows from  Theorem 3.2 in \cite{CS}.

\end{rk}

The proof of Theorem \ref{pg1} uses a deep result of S. Papadima and A. Suciu  \cite{PS} on resonance varieties with coefficients in a finite field  and a vanishing result of D. Cohen, A. Dimca and P. Orlik \cite{CDO}, obtained via perverse sheaves as explained in detail in \cite{ST} .
As a corollary, we apply this Theorem \ref{pg1} to the braid arrangement to recover S. Settepanella and A. M\u{a}cinic and S. Papadima results in this particular case.\\
We would like to thank A. M\u{a}cinic and G. Dehnam for some useful suggestions to improve the first version of this paper. Special thanks are due to A. Dimca for his help.

\section{Demonstration of Theorem \ref{pg1}}

The first result explains the role played by the graph $G(\A)$ in this story.

\begin{lem}\label{lem1}
Suppose the graph $G(\A)$ is connected. Then $H^1(A^*_R(\A), \omega_1 \wedge) =0\,\,$ for any unitary commutative ring $R.$
\end{lem}

\proof

We show that $ker\{\xymatrix{
A_R^1  \ar[r]^{\omega_1 \wedge} & A_R^{2} 
}\}= R. \omega_1.$
Let $b={\sum_{H \in \A}b_H a_H} $ be an element of $ A^1_R(\A).$ For all $X \in L(\A)$ we note ${\omega_1}_{X}={\sum_{H| X\subset H}a_H},\,$ and $b_X={\sum_{H| X\subset H}b_H a_H}.$\\
We have that $A_R^2(\A)= \displaystyle{\bigoplus_{X \in L_2(\A)} {A^2_R(\A_X)}},$ with $L_2(\A)=\{X\in L(\A)\,|\,codim\,X =2\}\,$ by Brieskorn decomposition theorem \cite{OT}. Hence we have:
\begin{center}
$\omega_1 \wedge b =0_R\,\, \Leftrightarrow\,\, {\omega_1}_{X} \wedge b_X = 0_R \,\,\,\forall X \in L_2(\A).$
\end{center}
So, suppose $\omega_1 \wedge b=0_R\,$ and let us show that $b_H= b_{H'}\,\,\forall H \neq H' \in \A.$\\
Let $H,\,H'$ be two distinct hyperplanes of $\A.$ Then $X= H \cap H' \in L_2(\A)$ and we have either:
\begin{itemize}
\item $H$ and $H'$ are linked with an edge and $\A_X=\{H,H'\}.$ In this case we have ${\omega_1}_{X} \wedge b_X = 0_R \Rightarrow (a_H+a_{H'})\wedge (b_Ha_H+b_{H'}a_{H'})=0_R \Rightarrow b_H=b_{H'}.$
\item Or $H$ and $H'$ are not directely linked by an edge but there exists $H_{i_1},\,...\,,H_{i_m}$ hyperplanes of $\A$ such that $H$ and $H_{i_1},\,\,H_{i_1}$ and $H_{i_2},\,...\,,H_{i_m}$ and $H'$ are linked. With the same considerations than in the first case we have that $b_H=b_{H_{i_1}},\,\,b_{H_{i_1}}=b_{H_{i_2}},\,...\,,b_{H_{i_m}}=b_{H'},\,$ so
$b_H=b_{H'}.$
\end{itemize}
Hence $b$ and $\omega_1$ are proportionnal and $H^1(A^*_R, \omega_1 \wedge) =0.$\\

\endproof
A shorter proof can be obtained using Lemma 3.3 in \cite{Lib-Yuz}, which says that if $\omega_1\wedge b=0,$ then for every $X=H \cap H' \in L_2(\A)$ such that $H$ and $H'$ are linked with an edge, we have that $b_H-b_{H'}=0.$ Then we can conclude with the connectivity of the graph (in this paper the result is stated only for a field $R$ of characteristic 0, but it is easy to see that it holds in general). For $R=\Z_p=\Z/p\Z,\,p\,$ prime, we can use  Lemma 4.9 of \cite{PM}, which is a generalization of Lemma 3.3 of \cite{Lib-Yuz} for finite fields, and we have that if $\omega_1 \wedge b=0,$ then for every $X=H \cap H' \in L_2(\A)$ such that $H$ and $H'$ are linked with an edge, we have: $b_H+b_{H'}=0\,$ if $p=2,\,$ and $b_H=b_{H'}$ if $p\neq 2.$ Hence we always have  $b_H=b_{H'} ,$ and we conclude with the connectivity of the graph.\\

\begin{rk}\label{rkomega}
In fact, a more general version of Lemma \ref{lem1} holds. If $\omega= \sum_{H \in \A} \omega_H a_H$ satisfies $\omega_H \neq 0_R \, \forall H \in \A,$ we can show that if $G(\A)$ is connected, then we have $H^1(A^*_R(\A), \omega \wedge) =0\,\,$ for any unitary commutative ring $R.$ Indeed, suppose $\omega \wedge b=0_R.$ If $H$ and $H'$ are linked by an edge, then $(\omega_H a_H + \omega_{H'} a_{H'})\wedge (b_H a_H + b_{H'} a_{H'}) = 0_R \\\Rightarrow\,\, \omega_H b_{H'} - \omega_{H'}b_{H}=0 \,\,\Rightarrow\,\, \exists t \in R\,\mid \, \left \{
\begin{array}{rcl}
b_H & = & t \omega_H \\
b_{H'} & = & t \omega_{H'}
\end{array}
\right.$.
If there exists $H_{i_1},\,...\,,H_{i_m}$ hyperplanes of $\A$ such that $H$ and $H_{i_1},\,\,H_{i_1}$ and $H_{i_2},\,...\,,H_{i_m}$ and $H'$ are linked, with the same considerations we have that there exists $t,t_1,\,...\,,t_m$ scalars in $R$ such that:
$\left \{
\begin{array}{rcl}
b_H & = & t \omega_H \\
b_{H_{i_1}} & = & t \omega_{H_{i_1}}
\end{array}
\right.$,
$\left \{\begin{array}{rcl}
b_{H_{i_1}} & = & t_1 \omega_{H_{i_1}} \\
b_{H_{i_2}} & = & t_1 \omega_{H_{i_2}} 
\end{array}
\right.$ $,\,...\,
,\left \{\begin{array}{rcl}
b_{H_{i_m}} & = & t_m \omega_{H_{i_m}} \\
b_{H'} & = & t_m \omega_{H'} 
\end{array}
\right.$. By identification we find that $t=t_1=\,...\,=t_m.$ Hence $b$ and $\omega$ are proportionnal. Furthermore, because $\A$ is central, we have that $H^*(M(\A),\Z)$ is torsion free. Now, let $u\in \C^*$ be a primitive root of the unity of order $p^s,\,p$ prime. Let $(k_H)_{H \in \A}$ be a collection of integers with g.c.d equal to 1 and such that $k_H \neq 0_{\Z_p}\,\,\forall\,H\in \A,$ where $\Z_p=\Z/p\Z.$ Let $\rho: \pi_1(M(\A))\rightarrow \C^*$ be the representation such that $\rho(\gamma_H)=u^{k_H},\,\,$ where $\gamma_H$ is the meridian around the hyperplane $H,$ and let $\LL$ be the associated rank one local system on $M(\A).$ Let us take $\omega= \sum_{H\in \A} [k_H]a_H \in A_{\Z_p}^1(\A),\,$ where $[k_H]$ is the mod $p$ reduction of $k_H.$ If $G(\A)$ is connected, then $H^1(A^*_{\Z_p}(\A), \omega \wedge) =0.$ Finally with Theorem 6.2 of \cite{PM} or Theorem C of \cite{PS} we have that $H^1(M(\A),\LL)=0.$ In fact it is generalization to the case $k_H=1\,\,\forall\,\,H\in \A,\,$ and $\omega=\omega_1$ that we will use in the proof of Theorem \ref{pg1}.  
\end{rk}

The second result is rather general, and we include it here for reader's sake, as we were not able to find a proper reference.

\begin{lem}\label{lem2}
Let  $\omega_1 \in A^*_R(\A)$ be as above and assume that $\partial \omega_1=d=0$ in $R$.
Then $H^1(A^*_R(\A), \omega_1 \wedge)=H^1(A^*_R(\A'), \omega_1 \wedge).$
\end{lem}

\proof
We have that $A^*_R(\A')=
 ker\{\xymatrix{
A^*_R(\A)  \ar[r]^{\partial} & A^*_R(\A) 
}\} \subset A^*_R(\A),\,$
so it is clear that $H^1(A^*_R(\A'), \omega_1 \wedge) \subset H^1(A^*_R(\A), \omega_1 \wedge).$ Now let $b \in ker\{\xymatrix{
A^1_R(\A)  \ar[r]^{\omega_1 \wedge} & A^2_R(\A) 
} \}.$ We have that $\partial(\omega_1 \wedge b) = d.b- \omega_1 .\partial b = -\omega_1 .\partial b =0 \,\,\Rightarrow\,\, \partial b =0$ in $R.$\\ Hence $b \in ker\{\xymatrix{
A^1_R(\A)  \ar[r]^{\partial} & A^0_R(\A) 
}\}= A^1_R(\A'),\,$ and we have that\\ $ ker\{\xymatrix{
A^1_R(\A)  \ar[r]^{\omega_1 \wedge} & A^2_R(\A) 
} \}\subset  ker\{\xymatrix{
A^1_R(\A')  \ar[r]^{\omega_1 \wedge} & A^2_R(\A') 
} \},\,$ and \\$H^1(A^*_R(\A), \omega_1 \wedge)\subset H^1(A^*_R(\A'), \omega_1 \wedge).$
\endproof

\begin{rk}\label{rk1}
If we take $H'_d$ the hyperplane at infty, we can define $A^*_R(\A')$ as the Orlik-Solomon algebra of the affine arrangement $\A'=\{H_1',\,...\,,H'_{d-1}\}\subset \C^n.$\\ Let $\omega_1'= \sum_{i=1}^{d-1}a'_{H_i} \in A^1_R(\A'),\,$ where $a'_{H_i}\in A^1_R(\A')$ denotes the element of $A^1_R(\A')$ corresponding to the hyperplane $H'_i.$ Then we have in fact that $a'_{H_i}= a_{H_i}-a_{H_d}$\\$\forall 1\leq i \leq d-1.$ So if  $R$ is finite field of characteristic $p,\,\,p$ prime, $\,p \mid d,\,$ then we have  $\omega_1'= \sum_{i=1}^{d-1}(a_{H_i} - a_{H_d})= \sum_{i=1}^{d-1}a_{H_i} - (d-1) a_{H_d}= \sum_{i=1}^{d}a_{H_i}= \omega_1.$ 
Hence $H^1(A^*_R(\A'), \omega'_1 \wedge)=H^1(A^*_R(\A'), \omega_1 \wedge).$
\end{rk}
Now we give the proof of  Theorem \ref{pg1}. For this we consider several cases.
\begin{enumerate}
\item If $6 |d,\,$ there exists $H \in \A$ such that:
\begin{center}
if $X\in L_2(\A),\,\,X \subset H,\,$ then $|\A_X| \neq 6.$
\end{center}
So the associated projective hyperplane $H'\in \A'$ is such that:
\begin{center}
if $X\in L_2(\A'),\,\,X \subset H',\,$ then $|\A'_X| \neq 6.$
\end{center}
Let $\lambda^k \ne 1$ be a $d$-th root of the unity.
The only edge of $\A'$ contained in $H'$ of codimension 1 is $H',$ and the corresponding monodromy operator of $\LL_{\lambda^k}$ is $T_{H'}= \lambda^k \neq 1$.
Let $X\in L(\A')$ be a dense edge of $\A'$ of codimension $2$ contained in $H'$. Then  the corresponding monodromy operator of  $\LL_{\lambda^k}$ about the divisor associated to $X$  is $$\,T_X= \lambda ^{k |\A'_X|},\,$$
 with $|\A'_X| \in \{1,2,3,4,5,7,8,9\}$, see \cite{ST}.
By using the vanishing result of Remark 2.4.20 of \cite{ST} applied to $\LL_{\lambda^k},$ we have that:
\begin{center}
$ord(\lambda^k) \notin \{1,2,3,4,5,7,8,9\}\Rightarrow H^1(F_\A)_{\lambda^k}=0.$
\end{center}
On the other hand, Theorem C of \cite{PS}, with our Lemmas \ref{lem1} and \ref{lem2} and Remark \ref{rk1} for a finite field $R= \Z_p=\Z/ p\Z$, show that if $ord(\lambda^k)=p^s$ with $p$ prime and $s\geq1,$ i.e. if
$ord(\lambda^k) \in \{2,3,4,5,7,8,9\}$, then again $ H^1(F_\A)_{\lambda^k}=0.$
Hence we have $H^1(F_\A,\C)= H^1(F_\A)_1$ in all the above subcases.

\item Suppose $6 \nmid d,$ and
let $X\in L(\A')$ be a dense edge of $\A'$ of codimension 2. With the same considerations as in the first subcase above, i.e. using  Remark 2.4.20 of \cite{ST},  we have that:
\begin{center}
$ord(\lambda^k) \notin \{1,2,3,4,5,6,7,8,9\}\Rightarrow H^1(F_\A)_{\lambda^k}=0.$
\end{center}
Because $ord(\lambda^k)| d,\,$ we have that $ord(\lambda^k)\neq 6,\,$ so $ord(\lambda^k) \in \{2,3,4,5,7,8,9\}$,
and we conclude as in the previous case, using Theorem C of \cite{PS}, and our Lemma \ref{lem1} .
\end{enumerate}

\section{Application}

We now apply Theorem \ref{pg1} to the braid arrangement $\A_n\subset \C^{n+1}$ with Milnor fiber $F_n.$ Recall that $\A_n$ is the collection of the hyperplanes 
\begin{center}
$H_{ij}:\,x_i-x_j=0,\,\,1\leq i<j\leq n+1.$
\end{center}

\begin{cor}\label{expg2}
We have that $H^1(F_{n},\C)=H^1(F_n)_1$ for any $ n \geq4.$
\end{cor}
\proof

Let us show that $G(\A_n)$ is connected for $n\geq 4.$

There are two types of  intersections $X \in L_2(\A_n)$ of codimension 2:
\begin{enumerate}
\item Type 1:\\
The intersections  $X = \{\, x_i=x_j,\,x_k=x_l,\,\,\,\, 1\leq i<j<k<l\leq n+1\,\,\},$\\
with corresponding subarrangement $\A_X=\{H_{ij},H_{kl}\}.$
\item Type 2:\\
The intersections $X = \{\, x_i=x_j=x_k,\,\,\,\, 1\leq i<j<k\leq n+1\,\},\,$ \\ with corresponding subarrangement $\A_X=\{H_{ij},H_{ik},H_{jk}\}.$
\end{enumerate}
Let $H_{ij}\,,\,H_{kl}\,,\,\,\,i<j\,,\, k<l\,\,$ be two distinct hyperplanes, \\and  $X=H_{ij}\cap\,H_{kl} \in L_2(\A_n).$ Suppose $i\leq k.$\\

\begin{itemize}

\item If $\{i,j\}\cap \{k,l\}= \emptyset,\,$ then $X$ is type 1 and $\A_X=\{H_{ij}\,,\,H_{kl}\}.$\\
Hence $H_{ij}$ and $H_{kl}$ are linked by an edge.\\

\item If $\{i,j\} \cap \{k,l\}\neq \emptyset$ then three cases are possible:
\vspace{0.2cm} 
\begin{enumerate}[(a)]
\item If $j=k,$ then the set $I=\{i,j,k,l\}$ has three elements.
Because $n \geq 4,$ the set $\{1,2,\,...\,,n+1\}$ has at least five elements and so contains two elements  $p<q$  such that $I \cap \{p,q\}= \emptyset.$
Hence $H_{ij} \cap H_{pq}$ and $H_{jl} \cap H_{pq}$ are two type 1 intersections and we have that  $H_{ij}$ and $H_{pq},\,$  and $H_{pq}$ and $H_{jl}$ are linked.
\item If $j=l,$ with the same considerations we have that there exists two element $p<q$ such that $H_{ij}$ and $H_{pq},\,$  and  $H_{pq}$ and $H_{kj}$ are linked.
\item If $i=k,$ with the same considerations we have that there exists two element $p<q$ such that $H_{ij}$ and  $H_{pq},\,$  and $H_{pq}$ and $H_{il}$ are linked.

\end{enumerate}
\end{itemize}
\vspace{0.2cm}
Finally $G(\A_n)$ is connected for $n\geq 4.$\\
Moreover, it is clear that $\A_n$ verifies the assumptions of  Theorem \ref{pg1} because $|{\A_n}_X|= 2$ or $3\,\,\,\,\forall X\in L_2(\A_n),\,$ and we have:
\begin{center}
$H^1(F_{\A_n},\C)=H^1(F_n)_1\,\,\,\forall n \geq4.$
\end{center}

\endproof
\begin{rk}\label{rkcor}
For $n=3,\,$ the graph $G(\A_3)$ has three connected components, so is not connected. It is known that
$H^1(F_3,\C)_{\ne 1}= H^1(F_{3})_{\lambda^2} \oplus H^1(F_{3})_{\lambda^4} $ is $2$-dimensional, \cite{PM}, \cite{S1}.
Similarly, for $n=2$, the graph $G(\A_2)$ has three connected components and $H^1(F_2,\C)_{\ne 1}= H^1(F_{2})_{\lambda} \oplus H^1(F_{2})_{\lambda^2} $,  is again $2$-dimensional, \cite{PM}, \cite{S1}.
For the Ceva arrangement given by
$$(x^3-y^3)(y^3-z^3)(x^3-z^3)=0,$$
 the graph $G(\A)$ has $9$ connected components (there are no edges in this case). It is known that
$H^1(F,\C)_{\ne 1}= H^1(F)_{\lambda^2} \oplus H^1(F)_{\lambda^4} $ is $4$-dimensional, see for instance \cite{BDS}.

Moreover, note that if $\A'$ is a line arrangement in $\PP^2$ coming from a pencil having $k\geq 3$ completely reducible fibers, see \cite{FY} , then the corresponding graph $G(\A)$ has at least $k$ connected components
\end{rk}

\begin{cor}\label{cor2} Assume that $\A'$ is a line arrangement in $\PP^2$ having only double and triple points. Assume that either

\noindent(i) the graph $G(\A)$ is connected, or

\noindent(ii) there is one line containing exactly one triple point and $d$ is even.\\
Then  $H^1(F_{\A},\C)=H^1(F_{\A})_1$.
\end{cor}
\proof
The case (i) follows directly from Theorem \ref{pg1}.\\
(ii) Let $H \in \A$ be the line containing exactly one triple point $p.$ Let $H_1,\,H_2 \in \A$ such that $\A_p=\{H,H_1,H_2\}.$ Every $H_i \notin \A_p$ is linked by an edge with $H$ and we have that $G(\A)$ is not connected if and only if $H_1$ or $H_2$ contains only triple points.
For example, if $H_1$ would contain only triple points, we could count the hyperplanes of $\A$ in the following manner: $H_1,(H_{i_1},H_{j_1}),\,(H_{i_2},H_{j_2}),\,...\,,(H_{i_{(d-1)/2}},H_{j_{(d-1)/2}}),\,$ where the pairs $(H_i,H_j)$ correspond to the points of multiplicity 3 contained in $H_1.$ Finally it would imply that $d = 2 .(d-1)/2 +1$ is odd, which is in contradiction with our assumptions. Hence $G(\A)$ is connected and we conclude directly with Theorem \ref{pg1}.
 
\endproof

Corollary \ref{cor2} is a direct consequence of Theorem \ref{pg1}. The next result (with a proof similar to the proof of Theorem \ref{pg1}) is more general, and can be obtained also as a consequence of Theorem 1.2 of \cite{Lib} saying that  if $H^1(F_{\A},\C) \neq H^1(F_{\A})_1,$ then $\A'$ comes from a pencil so $G(\A)$ is not connected. Indeed, for such a pencil, (ii) of Proposition \ref{proplib} is not verified except for $d=3.$ 

\begin{prop}\label{proplib}
Assume that $\A'$ is a line arrangement in $\PP^2$ having only double and triple points. Assume that either

\noindent(i) the graph $G(\A)$ is connected, or

\noindent(ii) $d=|\A')|>3$ and there is one line containing exactly one triple point.\\
Then  $H^1(F_{\A},\C)=H^1(F_{\A})_1$.
\end{prop}

\begin{proof}The case (i) follows directly from Theorem \ref{pg1}.\\
(ii) Let $H' \in \A'$ be the line containing exactly one triple point $p.$ Let $H'_1,\,H'_2 \in \A'$ such that $\A'_p=\{H',H'_1,H'_2\}.$
Let $\lambda^k \ne 1$ be a $d$-th root of the unity.
The only edge of $\A'$ of codimension 1 contained in $H'$  is $H',$ and the corresponding monodromy operator of $\LL_{\lambda^k}$ is $T_{H'}= \lambda^k \neq 1$. The dense edges of codimension 2  contained in $H'$  are the points of multiplicity 3 in $H'.$ So let $X \subset H'$ be such 
a dense edge. Then  the corresponding monodromy operator of  $\LL_{\lambda^k}$ about the divisor associated to $X$  is $\,T_X= \lambda ^{3k}.$
By using the vanishing result of Remark 2.4.20 of \cite{ST} applied to $\LL_{\lambda^k},$ we have that $ord(\lambda^k) \neq 3 \Rightarrow H^1(F_\A)_{\lambda^k}=0.$\\
Now let us show that $H^1(A^*_{\Z_3}(\A), \omega_1 \wedge)=0.$\\
 Let $b= \sum_{H \in \A} b_H a_H \in ker\{\xymatrix{
A_{\mathbb{\Z}_3}^1  \ar[r]^{\omega_1 \wedge} & A_{\mathbb{\Z}_3}^{2} 
}\}.$ 
If $H_k\in \A \backslash \A_p,\,$ then  $X=H \cap H_k$ is such that $\A_X=\{H,H_k\}$ and with the same considerations as in the proof of Lemma \ref{lem1} we have ${\omega_1}_{X} \wedge b_{X}=0 \Leftrightarrow b_{H_k}= b_H.$
We will show that $b_{H_1} = b_{H_2} = b_H.$ Let $H_k\in \A \backslash \A_p,\,$ and $X_1= H_1 \cap H_k\,,\,\,\,X_2= H_2 \cap H_k\,,\,$ the corresponding intersections with $H_1$ and $H_2.$ We consider several cases.
\begin{itemize}
\item If $|\A_{X_1}|=2$ and $|\A_{X_2}|=2,\,$ then  ${\omega_1}_{X_1} \wedge b_{X_1}=0\,$ and $\,{\omega_1}_{X_2} \wedge b_{X_2}=0 $\\$\Leftrightarrow b_{H_1}=b_{H_k}$ and $b_{H_2}=b_{H_k}.$ As $b_{H_k}= b_{H}$ we obtain $b_{H_1} = b_{H_2} = b_H.$
\item If  $|\A_{X_1}|=3$ and $|\A_{X_2}|=2,\,$ then $\A_{X_1}=\{H_1,H_k,H_j\},\,\,$ with $H_j \neq H, H_2$ (If $H_j = H$ or $H_2\,,$ then $X_1=p$ and $H_k \in \A_p$), and $\A_{X_2}=\{H_2,H_k\}.$\\
By taking $\{a_{H_1}a_{H_k} , a_{H_1}a_{H_j}\}$ as base of $A^2_{\mathbb{\Z}_3}(\A)$ whe have that\\
${\omega_1}_{X_1} \wedge b_{X_1}=0 \Leftrightarrow a_{H_1}a_{H_k}(2b_{H_k}-b_{H_1}-b_{H_j}) + a_{H_1}a_{H_j}(2b_{H_j}-b_{H_1}-b_{H_k})=0$ $\Leftrightarrow b_{H_1} + b_{H_k}+b_{H_j}=0\,\,(*)\,\,$ (in $\mathbb{Z}_3$). With $b_{H_k}= b_H$ and $b_{H_j}= b_H,\,$ we obtain $(*) \Leftrightarrow b_{H_1}=b_H.$ We also have $ {\omega_1}_{X_2} \wedge b_{X_2}=0 \Leftrightarrow b_{H_2}=b_{H_k}.$ Finally $b_{H_1} = b_{H_2} = b_H.$
\item If  $|\A_{X_1}|=3$ and $|\A_{X_2}|=3,\,$ then $\A_{X_1}=\{H_1,H_k,H_j\},\,\,$ with $H_j \neq H, H_2$\\ and  $\A_{X_2}=\{H_2,H_k,H_l\},\,\,$ with $H_l \neq H, H_1.$ With the same considerations as in the previous case we have ${\omega_1}_{X_1} \wedge b_{X_1}=0 \Leftrightarrow  b_{H_1} + b_{H_k}+b_{H_j}=0\,\,(*)$ and ${\omega_1}_{X_2} \wedge b_{X_2}=0 \Leftrightarrow  b_{H_2} + b_{H_k}+b_{H_l}=0\,\,(**).$ With $b_{H_k}= b_{H_j}= b_{H_l}= b_H,\,$ we obtain $(*) \Leftrightarrow b_{H_1}=b_H\,$ and $(**) \Leftrightarrow b_{H_2}=b_H.$ Finally $b_{H_1} = b_{H_2} = b_H.$
\end{itemize}
Hence $b$ and $\omega_1$ are proportional and $H^1(A^*_{\Z_3}(\A), \omega_1 \wedge)=0.$\\
On the other hand, Theorem C of \cite{PS}, with our Lemmas \ref{lem1} and \ref{lem2} and Remark \ref{rk1} for $R= \Z_3$ gives: $ord(\lambda^k)= 3 \Rightarrow H^1(F_{\A})_{\lambda^k}=0.$ 
\end{proof}

The following examples show the difficulty of the problem in the  general case.
First we give an example where $G(\A)$ is not connected, showing that the conditions (ii) and (iii) in Theorem \ref{pg1} are not sufficient.

\begin{ex}\label{ex1}
Let $\A'\subset \PP^2$ be the arrangement defined by the homogeneous polynomial $Q(x:y:z)= xyz(x^4-y^4)(y^4-z^4)(x^4-z^4).$ Lines of $\A'$ are: $\{x=0\},\{y=0\},$\\$\{z=0\},\,d_1,d_2,d_3,d_4,d_5,d_6,d_7,d_8,d_9,d_{10},d_{11},d_{12},\,$ where $d_1,d_2,d_3,d_4$ are of the form $x= \alpha y,\,$  $d_5,d_6,d_7,d_8$ are of the form $y= \alpha z,\,$ and  $d_9,d_{10},d_{11},d_{12}$ are of the form $x= \alpha z,\,$ with $\alpha^4=1.$ The intersection between $d_1,d_2,d_3,d_4$ with $\{z=0\}$ is a point of multiplicity 2, and it is the same for $d_5,d_6,d_7,d_8$ with $\{x=0\},\,$ and for $d_9,d_{10},d_{11},d_{12}$ with $\{y=0\}.$ The other intersection lines of $\A'$ are points of multiplicity 3 or 6. Indeed, if we take $i \in \{1,2,3,4\}$ and $j \in \{5,6,7,8\},$ we have  $d_i \cap d_j:=\{ x= \alpha_i y\} \cap \{ y= \alpha_j z\}$ with $(\alpha_i \alpha_j)^4=1.$ So, $d_i \cap d_j= d_i \cap d_j \cap d_k$ where $d_k:=\{x=\alpha_i \alpha_j z\},\,k \in \{9,10,11,12\}.$ Similary we have a point of multiplicity 3 if we take $i\in\{1,2,3,4\}$ and $j \in \{9,10,11,12\},$ and if we take $i\in\{5,6,7,8\}$ and $j \in \{9,10,11,12\}.$ Then we have tree points of multiplicity 6: \\$d_1 \cap d_2 \cap d_3 \cap d_4 \cap \{x=0\} \cap \{y=0\},\, d_5\cap d_6 \cap d_7 \cap d_8 \cap \{y=0\} \cap \{z=0\},\, $ and $d_9\cap d_{10} \cap d_{11} \cap d_{12} \cap\{x=0\} \cap \{z=0\}.$ Hence $G(\A)$ has three connected components and is not connected.  It's clear that $\A$ verifies points (ii) and (iii) of Theorem \ref{pg1}, and we have that $H^1(F_{\A},\C) \neq H^1(F_{\A})_1$ , see  Remark 3.4 (iii) of \cite{DP}.

\end{ex}

 When the assumptions of Theorem \ref{pg1} are not verified, it's very complicated to conclude and we have to use other  results.\\

\begin{ex}\label{ex2}
Let $\A'\subset \PP^2$ be the arrangement defined by the homogeneous polynomial $Q(x:y:z)= xyz(x^2-y^2)(y^2-z^2)(x^2-z^2).$ With the same arguments as in Example \ref{ex1} we can show that $G(\A)$ is not connected and $\A$ verifies points (ii) and (iii) of Theorem \ref{pg1}. But here we have that $H^1(F_{\A},\C) = H^1(F_{\A})_1$, see \cite{CS}  or Remark 3.4 (ii) of \cite{DP}.
\end{ex}

\begin{ex}\label{ex3}
Let  $\A'\subset \PP^2$ be the arrangement defined by the homogeneous polynomial $Q(x:y:z)= xy(x+y)(x-y)(x+2y)(x-2y)(2x+y+z)(2x+y+2z)(2x+y+3z)(2x+y-z)(2x+y-2z)(2x+y-3z).$ Here $d=12$ and we have two intersections in $L_2(\A)$ of multiplicity 6: $\{x=y=0\},\,$ and $\{y=-2x\} \cap \{z=0\}.$ One can easily verify that each hyperplane contains  one of these two intersections and that $G(\A)$ is connected. Indeed, any hyperplane in $\{\{x=0\},\{y=0\},\{x+y=0\},\{x-y=0\},\{x+2y=0\},\{x-2y=0\}\}$ is linked by an edge with any hyperplane in  $\{ \{2x+y+z=0\},\{2x+y+2z=0\},\{2x+y+3z=0\},\{2x+y-z=0\},\{2x+y-2z=0\},\{2x+y-3z=0\}\}.$ Hence (i) and (ii) of Theorem \ref{pg1} are verified, but not (iii).\\ The minimal number of lines in $\A'$ containing all the points of multiplicity at least 3 is 2, so with Theorem 1.1 of \cite{SHA} we have that $\A'$ belongs to the class $\mathcal{C}_2,$ and any rank one local system on $M(\A')$ is admissible. Hence if we take $\lambda^k\neq 1,\,$ there exists 
$$\omega= \sum_{H \in \A'} \omega_H \frac{dl_H}{l_H} \in H^1(M(\A'),\C),\,$$
 (where $l_H$ is the linear form defining the hyperplane  $H$) such that 
$$\dim H^1(F_{\A})_{\lambda^k}=\dim H^1(M(\A'),\LL_{\lambda^k})= \dim H^1(H^*(M(\A'),\C),\omega\wedge).$$
 Furthermore it is known that $exp(2\pi \sqrt{-1}\,\omega_H)=\lambda^k \,\,\forall H,\,$ so $\omega_H\neq 0 \,\,\forall H.$\\
Assume  $\dim H^1(H^*(M(\A'),\C),\omega\wedge)\neq0.$ Then $\omega\in\mathcal{R}_1(\A')=\{\alpha \in  H^1(M(\A'),\C)|$\\$\dim H^1(H^*(M(\A'),\C),\alpha\wedge)\geq1\}.$ With the description of the irreductible components of the first resonance variety for a $\mathcal{C}_2$ arrangement, (see Theorem 4.3 of \cite{Dinh}) we have a contradiction with the fact that $\omega_H\neq 0 \,\,\forall H.$ \\Hence $H^1(H^*(M(\A'),\C),\omega\wedge)=0,$ and we have that $H^1(F_{\A},\C)=H^1(F_{\A})_1.$

\end{ex}

\begin{ex}\label{ex4}
Let  $\A\subset \C^4$ be the arrangement defined by the homogeneous polynomial $Q(x,y,z,t)= xy(x-y)(x+y)(x-2y)(x+2y)zt(z-t)(z+t)(z-2t)(z+2t).$ Here $d=12$ and we have two intersections in $L_2(\A)$ of multiplicity 6: $\{x=y=0\},\,$ and $\{z=t=0\}.$ One can easily verify that each hyperplane contains  one of these two intersections and that $G(\A)$ is connected. Indeed, any hyperplane in $\{ \{x=0\},\{y=0\},\{x-y=0\},\{x+y=0\},\{x-2y=0\},\{x+2y=0\}\}$ is linked by an edge with any hyperplane in  $\{ \{z=0\},\{t=0\},\{z-t=0\},\{z+t=0\},\{z-2t=0\},\{z+2t=0\}\}.$ Hence (i) and (ii) of Theorem \ref{pg1} are verified, but not (iii).\\
One can decompose $\A$ in two arrangements with distinct variables: $\A=\A_1 \times \A_2,\,$ where $\A_1$ is defined by $Q_1(x,y)= xy(x-y)(x+y)(x-2y)(x+2y),\,$ and $\A_2$ is defined by $Q_2(z,t)= zt(z-t)(z+t)(z-2t)(z+2t).$ Let us take $\lambda^k\neq 1$ and note $F_1$ and $F_2$ the Milnor fibers of the subarrangements $\A_1$ and $\A_2$ in $\C^2$. Then applying Theorem 1.4 (i) of \cite{DNA} we have that $H^1(F_{\A})_{\lambda^k}= (H^*(\mathbb{T},\C) \oplus H^*(F_1)_{\lambda^k} \oplus H^*(F_2)_{\lambda^k})^1=0 ,$ so $H^1(F_{\A},\C)=H^1(F_{\A})_1.$
\end{ex}

In fact, a more general version of Corollary \ref{expg2} holds.\\
Let $\Gamma$ be a simple graph (that is to say it contains no loop and no double edge) and $\A_{\Gamma}$ be the corresponding graphic arrangement, see for instance \cite{PM}. In other words, $\A_{\Gamma}$ is a subarrangement of the braid arrangement.
Such a graph $\Gamma$ is composed by edges $(ij),\,i< j,\,$ and the corresponding arrangement $\A_{\Gamma}$ is composed by the hyperplanes $H_{ij}: x_i-x_j=0,\,\,(ij)\in \Gamma.$ We will note $|\Gamma|$ the number of vertices of $\Gamma.$ We say that $\Gamma$ is connected if we can link two different vertices with an edge sequence.
Let $\omega_1= \sum_{(ij)\in \Gamma} a_{ij} \in A^1_R(\A_{\Gamma}).$

 We have the following result:

\begin{lem}\label{lemgraph}
Suppose $\Gamma$ is connected and $|\Gamma|\geq 5,$ then $H^1(A^*_R(\A_{\Gamma}), \omega_1 \wedge) =0\,\,$ for any unitary commutative ring $R.$
\end{lem}

\proof Let $b=\sum_{(ij)\in \Gamma}b_{ij} a_{ij} \in A^1_R(\A_{\Gamma})$ such that $\omega_1 \wedge b=0.$ Let us show that $b$ and $\omega_1$ are proportional. With Lemma 3.3 of  \cite{Lib-Yuz} or Lemma 4.9 of \cite{PM} we have:\\
\noindent(i) For all intersection of type 1: $H_{ij} \cap H_{kl},\,i<j<k<l,\,$ we have $b_{ij}=b_{kl}.$\\
\noindent(i) For all intersection of type 2: $ H_{ij} \cap H_{ik} \cap H_{jk},\,i<j<k,\,$ we have either:\\ $b_{ij}=b_{ik}=b_{jk}$ if $3 \neq 0_R,\,$ or $b_{ij}+b_{ik}+b_{jk}=0$ if $3 = 0_R.$\\
So let us suppose that $3=0_R,\,$ and let us take an intersection of type 2: $H_{ij} \cap H_{ik} \cap H_{jk},\,i<j<k.$ We will show that $b_{ij}=b_{ik}=b_{jk}.$\\
Because  $|\A_{\Gamma}|\geq 5,$ there exists two additional vertices $s$ and $m,$ so there exists two additional edges and because $\Gamma$ is connected, these two edges are linked either: with two different vertices of the triangle $ijk,\,$ or with one of the vertices of the triangle $ijk,\,$ or one of these edges is linked with a vertice of the triangle $ijk\,$ and the other is linked with the new vertice of the first one. By symmetry, we can assume we are in one of the following cases  (here we have chosen $i<j<k<s<m$ but of course the order does not matter):
\begin{enumerate}
\item If $(is),\,(im)\in \Gamma:$  
\begin{enumerate}
\item If $(js),\,(km)\notin \Gamma,\,$ then $H_{ij}-H_{is}-H_{jk}-H_{im}-H_{ik},\,$ and $b_{ij}=b_{ik}=b_{jk}.$
\item If $(js),\,(km)\in \Gamma,\,$ then $H_{ij}-H_{km}-H_{js}-H_{ik}-H_{js}-H_{im}-H_{jk},\,$ and $b_{ij}=b_{ik}=b_{jk}.$
\item If $(js)\in \Gamma,\,(km)\notin \Gamma,\,$ here the graph $G(\A_{\Gamma})$ is not necessarily connected (for example it is not connected if $(jm) \in \Gamma$). But we have \\ $H_{ik}-H_{im}-H_{jk},\,$ so $b_{ik}=b_{jk}$ and  with $b_{ij}+b_{ik}+b_{jk}=0$ we have $b_{ij}=-2b_{ik}=b_{ik}\in R.$
\item If $(js)\notin \Gamma,\,(km)\in \Gamma,\,$ it's the symetric case of the previous.
\end{enumerate}
\item If $(is),\,(jm)\in \Gamma,\,$ the graph $G(\A_{\Gamma})$ is not necessarily connected. We have $H_{ik}-H_{jm}-H_{is}-H_{jk},$ so $b_{ik}=b_{jk}$ and $b_{ij}+b_{ik}+b_{jk}=0$ implies $b_{ij}=b_{ik}.$
\item If $(is),\,(sm)\in \Gamma,\,$ then $H_{ij},H_{ik}$ and $H_{jk}$ are linked with $H_{sm}$ in $G(\A_{\Gamma})$ and we conclude directly.
\end{enumerate}

\endproof

\begin{rk}\label{rkgraph}
\noindent(i) We have that $\Gamma$ is connected and $|\Gamma|\geq 5$ does not imply $G(\A_{\Gamma})$ is connected:
 the graphic arrangement $\A_{\Gamma}=\{H_{12},H_{13},H_{14},H_{15},H_{23},H_{34},H_{35}\}$ verifies $\Gamma$ is connected and $|\Gamma|\geq 5,$ but $G(\A_{\Gamma})$ is not connected ($H_{13}$ is linked with any hyperplane of $\A_{\Gamma})$ and we can't apply our Lemma \ref{lem1}.\\
\noindent(ii) If  $\Gamma$ is connected and $|\Gamma|\geq 5,$ then we recover A. M\u{a}cini and S. Papadima results for graphic arrangements \cite{PM}. Indeed, with similar considerations as in the proof of Theorem \ref{pg1} and by using Lemma \ref{lemgraph} instead of Lemma \ref{lem1}, we have that $H^1(F_{\A_{\Gamma}},\C)=H^1(F_{\A_{\Gamma}})_1.$\\
If the graph  $\Gamma$ is not connected, but each of its connected components $\Gamma_i$ satisfies $|\Gamma_i|\geq 5,$ then $\A_{\Gamma}$ is a product of arrangements $\A_{\Gamma_i}$ and we can conclude using Theorem 1.4 (i) of \cite{DNA}.
\end{rk}

\end{document}